\documentclass[11pt, reqno]{article}
\usepackage{a4wide}
\usepackage{tocloft}
\usepackage{amssymb}
\usepackage{amsmath}
\usepackage{amsthm}
\usepackage{tikz}
\usepackage{amscd}
\usepackage{tikz-cd}
\usetikzlibrary{cd}

\usepackage[curve]{xypic}
\usepackage{hyperref}
\usepackage[all]{xy}
\usepackage{color}
\theoremstyle{plain}
\usepackage[margin=1in]{geometry}

\newcommand{\id}{\operatorname{id}}

\newcommand{\sch}[1]{\operatorname{{\bf  #1}}}
\newcommand{\integers}[1]{\operatorname{\mathfrak{o}_{{\it #1}}}}
\newcommand{\ideal}[1]{\operatorname{\mathfrak{p}_{{\it #1}}}}
\newcommand{\mul}{\operatorname{\mathbb{G}_{\text{m}}}}
\newcommand{\ind}{\operatorname{ind}}

\newcommand{\diag}{\operatorname{diag}}

\newcommand{\End}{\operatorname{End}}
\newcommand{\Ext}{\operatorname{Ext}}

\newcommand{\bv}{\operatorname{{\bf v}}}

\newtheorem{theorem}{Theorem}[section] 
\newtheorem*{theorem*}{Theorem}
\newtheorem{corollary}[theorem]{Corollary}
\newtheorem{lemma}[theorem]{Lemma}
\newtheorem{remark}[theorem]{Remark}
\newtheorem{proposition}[theorem]{Proposition}

\newtheorem{example}[theorem]{Example}

\newtheorem{hypothesis}{Hypothesis}[section]
 
\author{\large{ Santosh Nadimpalli }} \date{\today} \title{On
  extensions of characters of affine pro-$p$ Iwahori--Hecke algebra.}
\begin{document}
\maketitle
\begin{abstract}
  Let $K$ be a non-discrete non-Archimedean local field with residue
  characteristic $p$. Let $G$ be the group of $K$ rational points of a
  algebraic connected reductive group defined over $K$. In this
  article we compute the extensions between characters of affine
  pro-$p$ Iwahori--Hecke algebra $\mathcal{H}^{\text{aff}}$ over an
  algebraically closed field $R$ of characteristic $p$. 
\end{abstract}
\section{Introduction}
Let $K$ be a non-discrete non-Archimedean local field and $\sch{G}$ be
a connected reductive group scheme over $K$ with an irreducible root
system.  In this article we are interested in computing extensions of
supersingular characters of affine pro-$p$ Iwahori--Hecke algebra,
denoted by $\mathcal{H}^{\text{aff}}$ (see section $2$). In the
context of mod-$p$ local Langlands correspondence the Iwahori--Hecke
algebra, denoted by $\mathcal{H}$, plays a very important role. For
instance, when $\sch{G}={\rm GL}_n/K$ a numerical correspondence
between absolutely simple supersingular $\mathcal{H}$-modules of
dimension $n$ and $n$-dimensional absolutely irreducible mod-$p$
representations of the absolute Galois group of $\mathbb{Q}_p$ was
conjectured by Vigneras (see \cite{vigneras_split}) and is proved by
Ollivier for ${\rm GL}_n(\mathbb{Q}_p)$ (see \cite[Theorem
1.1]{ollivier_numerical}). This numerical correspondence is extended
as an exact functor by the work of Grosse-Kl\"onne (see \cite[Theorem
8.8]{grosse_klonne_hecke_galois_1}). This article is an attempt to
understand the blocks of the category of $\mathcal{H}$-modules.

We compute the dimension of degree one extensions of characters of the
affine pro-$p$ Iwahori--Hekce algebra for a connected reductive group
over a $p$-adic field. With the recent work of Abe on parabolic
induction and their adjoint functors (see \cite{abe_hecke_induct_1}
and \cite{abe_hecke_induct_2}) we know the dimension of
degree one extensions between simple modules of
$\mathcal{H}^{\text{aff}}$. In particular if the group is semi-simple
and simply connected the dimension of degree one extensions between
simple modules of pro-$p$ Iwahori--Hecke algebra, denoted by
$\mathcal{H}$, can be computed using the data of support of the
supersingular characters (see section $2$).

The structure of the algebra $\mathcal{H}$ is studied in detail by
Vigneras in the article \cite{vigneras_split} when $\sch{G}$ is split
and in the article \cite{vigneras_pro-p_hecke_alg} for any connected
reductive group $\sch{G}$. Among other things she showed presentations
of the algebra $\mathcal{H}$ similar to that of Iwahori--Matsumoto and
Bernstein as in the classical Iwahori--Hecke algebra. Ollivier defied
the notion of supersingularity for a $\mathcal{H}$ module in the split
case (see \cite{ollivier_satake_comp}) and later obtained a
classification of simple supersingular modules of $\mathcal{H}$. These
results were generalised to arbitrary groups by Vigneras (see
\cite{vigneras_modules}). All simple modules are constructed from
supersingular representations (see \cite{ollivier_induct} and
\cite{abe_hecke_induct_1}). The simple supersingular modules are
characterised completely by the results of Ollivier and Vigneras and we
use their explicit description in our calculations.

The dimensions of extensions spaces of simple modules over
$\mathcal{H}$ are computed for ${\rm GL}_2/K$ by Breuil and Paskunas
(see \cite{towards_mop_p_langlands}). The homological dimension of the
algebra $\mathcal{H}$ is investigated by Koziol in the article
\cite{koziol_hom_dim}. When $\sch{G}$ is split group Koziol showed
that the homological dimension is usually infinite. In
this regard the higher extensions always exist but the question of
blocks for the category of modules over $\mathcal{H}$ is to be
determined. The work of Abdellatif and the author gives dimensions of
extension spaces between simple modules of $\mathcal{H}$ for
${\rm SL}_2$. For rank one groups, it was observed that the notion of
$L$-packets and the blocks are closely related. To understand the
blocks of Iwahori--Hecke modules we first compute the dimension of
degree one extensions (see Theorem \ref{main_theorem}). For the case
of ${\rm SL}_n$ the defintion of $L$-packets was given by Koziol (see
\cite[Definition 6.4]{koziol_sl_packets}). We observed that the
notions of supersingular blocks and $L$-packets do not coincide if
$n>2$.

In this article we explicitly compute the blocks of simple
$\mathcal{H}$ modules for unramified unitary groups in $2$ and $3$
variables denoted by $U(1,1)$ and $U(2,1)$ respectively. The case of
$U(1,1)$ is similar to that of ${\rm SL}_2$ and blocks and $L$-packets
are the same. This paper is the authors attempt to understand the
relationship between blocks and $L$ packets for general reductive
groups. Even for the case of ${\rm SL}_n$ the complete relationship
between blocks and $L$-packets for $n>3$ is not complete.
\section{Preliminaries}
Let $K$ be a non-Archimedean non-discrete local field with ring of
integers $\integers{K}$, its maximal ideal $\ideal{K}$ and residue
field $k$ of cardinality $q$ a power of prime $p$. In this article all
modules and representations are over a fixed algebraically closed
field $R$ of characteristic $p$. Let $\sch{G}$ be a connected
reductive group scheme over $K$. We denote by $\mathfrak{X}_K$ the
adjoint Bruhat--Tits building associated to $(\sch{G}, K)$ and
$j:\mathfrak{X}_K\hookrightarrow \mathfrak{X}_K'$ be the enlarged
Bruhat--Tits building. For any facet $F$ of $\mathfrak{X}_K$ we denote
by $\sch{G}_F$ the Bruhat--Tits group scheme over $\integers{K}$
associated to $F$ such that $\sch{G}_F(\integers{K})$ is isomorphic to
the pointwise $G$ stabiliser of the facet $j(F)$ and
$\sch{G}_F\times_{\integers{K}}K\simeq \sch{G}$. Let $\sch{G}_F^0$ be
the connected component of $\sch{G}_F$ and let $P_F$ be the group
$\sch{G}_F^0(\integers{K})$. The group $P_F$ is the parahoric subgroup
of $\sch{G}(K)$. Let $U_{F, k}$ be the unipotent radical of
$\sch{G}^0\times_{\integers{K}} k$. Let $I_F$ be the pro-$p$ group
$$\{\sch{G}^0_F(\integers{K})\ |\ g\in U_{F, k}(k)\ \text{mod}\
\ideal{K}\}.$$
For any $K$-group scheme $\sch{H}$ we denote by $H$ the group
$\sch{H}(K)$. 

In the rest of the section the main reference is
\cite[section 1.3]{vigneras_modules}. Let $\sch{T}$ be the maximal $K$-split
torus contained in $\sch{G}$. We denote by $\sch{N}$ and $\sch{Z}$ 
the normaliser and centraliser of the torus $\sch{T}$. Let $V$ be the
space spanned by the set of coroots
$\Phi(\sch{G}, \sch{T})^{\vee}\subset
X_\ast(\sch{T})\otimes\mathbb{R}$. Let $\mathcal{A}_{\sch{T}}$ be the
apartment in $\mathfrak{X}_K$ corresponding to $\sch{T}$ and
$\nu:Z\rightarrow V$ be the Bruhat--Tits homomorphism. The group $Z$
acts on $\mathcal{A}_{\sch{T}}$ by translations via the map
$\nu:Z\rightarrow V$ moreover this action extends to an action of $N$.
Let $Z_0$ be the unique parahoric subgroup of $Z$ and $Z_1$ be the
maximal pro-$p$ subgroup of $Z_0$. We denote by $Z_k$ the quotient
$Z_0/Z_1$. Let $W(1)$ be the group $N/Z_1$ and $W$ be the extended
affine Weyl group $N/Z_0$. With these notations $W(1)$ fits in the
following exact sequence.
$$0\rightarrow Z_k\rightarrow W(1)\rightarrow W\rightarrow 0.$$
Let $\Lambda$ be the group $Z/Z_0$, the group $W_0$ normalizes $Z/Z_0$
and we have an isomorphism $W\simeq \Lambda\rtimes W_0$. The
homomorphism $\nu:Z\rightarrow V$ factorizes through $Z_0$ and hence
$W$ acts on $V$.

We fix a chamber $C$ contained in $\mathcal{A}_{\sch{T}}$. The group
$P_C$ is the Iwahori subgroup and $I_C$ is the maximal pro-$p$
subgroup of $P_C$. The maps $n\mapsto P_CnP_C$ and $n\mapsto I_CnI_C$
induces bijections $W\simeq P_C\backslash G/P_C$ and
$W(1)\simeq I_C\backslash G/I_C$ respectively. Let $\mathcal{S}(C)$ be
the set of faces of $\bar{C}$ and for every $F\in \mathcal{S}(C)$ we
denote by $s_F$ the affine reflection fixing the face $F$. Let
$W^{\text{aff}}$ be the group generated by
$\{s_F\ |\ F\in \mathcal{S}(C)\}$, we denote this set of reflations by
$S^{\text{aff}}$. The group $W^{\text{aff}}$ is called the affine Weyl
group. The pair $(W^{\text{aff}}, S^{\text{aff}})$ is a Coxeter system
and the group $W^{\text{aff}}$ is contained in $\nu(N)$ the image of
$N$ in the group of affine automorphisms of
$\mathcal{A}_{\sch{T}}$. Let $l:W^{\text{aff}}\rightarrow \mathbb{Z}$
be the length function of the Coxeter system
$(W^{\text{aff}}, S^{\text{aff}})$.

For any $F\in \mathcal{S}(C)$ we denote by $G_{F,k}$ the group
$P_F/I_F$. The group $P_{F, k}$ is also the $k$-points of the
reductive quotient of $\sch{G}^0\times_{\integers{K}}k$, denoted by
$\sch{G}_{F, k}$. The group $\sch{G}_{F, k}$ is a connected reductive
group of rank one. The image of $I_F$ in $P_{F, k}$ is the group of
$k$-rational points of the unipotent radical $\sch{U}_{F, k}$ of a
Borel subgroup ${T}_{F, k}{U}_{F, k}$. We denote by
$\overline{\sch{U}}_{F, k}$ be the unipotent radical of the opposite
Borel subgroup of $\sch{G}_{F, k}$. We denote by $Z_{F, k}$ the group
$Z_k\cap<U_{F, k}, \overline{U}_{F,k}>$ where
$<U_{F, k}, \overline{U}_{F, k}>$ is the group generated by the two
opposite unipotent groups. Moreover for any $s\in S^{\text{aff}}$
there exits an $n_s\in N\cap P_s$ such that its image in $G_{k,s}$
belongs to the group $<U_{F, k}, \overline{U}_{F, k}>$. The image of
$n_s$ in $W(1)$ is called an admissible lift of $s$.

Let $\Omega$ be the $W$ stabiliser of $C$. The group $W$ can be
identified with $W^{\text{aff}}\rtimes \Omega$.  The group $\Omega$
normalizes $W^{\text{aff}}$ and the length function $l$ extends to a
function on $W$. We denote by $l$ the inflation of $l$ to $W(1)$. If
$\sch{G}$ is semi-simple simply-connected group $\Omega$ is
trivial. The group $\Omega$ is trivial in some other interesting
cases. Consider an unramified quadratic extension $L$ of $K$ and
$(W, h)$ be a pair consisting of a vector space $W$ over $L$ and $h$
is a hermitian form. Assume that the dimension of the anisotropic part
of $W$ is less than one. The unitary group $\sch{U}(W)/K$ associated
to the pair $(W, h)$ is quasi-split and in this case
$W=W^{\text{aff}}$. We may take $L$ to be a separated ramified
quadratic extension of $K$, when the dimension of $W$ is odd and we
have $W=W^{\text{aff}}$ for $\sch{U}(W)/K$.

Let $\mathcal{H}$ be the algebra $\End_{G}(\ind_{I_C}^G\id)$ and we
identify $\mathcal{H}$ with the space of functions $f:G\rightarrow R$
such that $f(i_1gi_2)=f(g)$ for all $i_1, i_2\in I_C$ and $g\in
G$. For any $w\in W(1)$ we denote by $T_w$ the characteristic function
on $I_CwI_C$. The elements $\{T_w\}_{w\in W(1)}$ form a basis for the
Hecke algebra and $\mathcal{H}$ admits the following presentation
given by two sets of relations
\begin{enumerate}
\item (braid relations) For any $w, w'\in W(1)$ such that
  $l(w)+l(w')=l(ww')$ we have $T_{w}T_{w'}=T_{ww'}$
\item (quadratic relations) For any $s\in S^{\text{aff}}$ we have
  $T_s^2=-c_sT_s$ where $c_s=1/|Z_{k,s}|\sum_{z\in Z_{k,s}}T_z$.
\end{enumerate} 
We denote by $\mathcal{H}^{\text{aff}}$ the subalgebra of
$\mathcal{H}$ generated by $\{T_{w}\ | \ w\in W^{\text{aff}}\}$. The
algebra $\mathcal{H}$ is called the pro-$p$ Iwahori--Hecke algebra and
$\mathcal{H}^{\text{aff}}$ the affine pro-$p$ Iwahori--Hecke
algebra. The algebra $\mathcal{H}$ is isomorphic to a certain twisted
tensor product of $R[Z_k]$ and $\mathcal{H}^{\text{aff}}$. In this
article we restrict to the characters of $\mathcal{H}^{\text{aff}}$
and we will not need this description. Let $\iota$ be an involutive
$R$ automorphism of $\mathcal{H}$ such that
$\iota(T_{\tilde{s}})=T_{\tilde{s}}-c_{\tilde{s}}$.

We restrict ourselves to characters of $\mathcal{H}$ and
$\mathcal{H}^{\text{aff}}$ we do not recall the description of all
simple modules. We first describe the set of characters of
$\mathcal{H}^{\text{aff}}$ (see \cite[Theorem
1.6]{vigneras_modules}). Let $\lambda$ be any character of $Z_k$ and
we denote by $S_\lambda$ the set
$\{s\in S^{\text{aff}}\ |\ \lambda(c_s)\neq 0\}$. For any subset $I$
of $S^{\text{aff}}$ we denote by $W_I$ the subgroup of $W$ generated
by $s\in I$. The set of characters of the algebra $\mathcal{H}$ are
parametrised by pairs $(\lambda, I)$ consisting of a character
$\lambda$ of $Z_k$ and a subset $I$ of $S_\lambda$. We denote by
$\xi_{\lambda, I}$ the character corresponding to $(\lambda, I)$ and
is given by
\begin{equation}
  \xi_{\lambda, I}(T_{wt})=0\ \text{for all}\ w\in W\backslash W_I\
  \text{and} \ t\in Z_k.
\end{equation}
\begin{equation}
\xi_{\lambda, I}(T_{wt})=\lambda(t)(-1)^{l(w)}\ \text{for all}\ w\in W_I\
\text{and} \ t\in Z_k.
\end{equation}
For any character $\xi$ of $\mathcal{H}^{\text{aff}}$ we denote by
$S_{\xi}$ the set $\{s\in S^{\text{aff}}\ |\ \ c_{s}(T_{\tilde{s}})\}$
where $\tilde{s}$ is an admissible lift of $s$. This definition of
$S_{\xi}$ does not depend on admissible lifts. The character $\xi$ is
called sign character if $S_{\xi}=S^{\text{aff}}$. If $\xi$ is a sign
character then $\xi\circ\iota$ is called the trivial character. Any
character $\xi$ is supersingular if and only if $\xi$ is not a sign
character or trivial character (see \cite[Theorem
1.6]{vigneras_modules}).

\section{Calculations of degree one extensions.}

In this section we want to compute the dimension of the spaces
$\Ext^1_{\mathcal{H}^{\text{aff}}}(\xi_{\lambda_1, I_1}, \xi_{\lambda_2, I_2})$. We
denote by $\mathcal{H}$ the affine pro-$p$ Iwahori--Hecke algebra by
abuse of notation. The algebra $\mathcal{H}$ is generated by $T_t$ for
$t\in Z_k$ and $T_{\tilde{s}}$ where $s\in S_{\text{aff}}$. For
convenience we drop the $\tilde{s}$ in the admissible lift of $s$. We
use the generators and relations to calculate the dimension of the
degree one extensions.

Let $E$ be an extension of $\mathfrak{m}:=\xi_{\lambda_2, I_2}$ by
$\mathfrak{n}:=\xi_{\lambda_1, I_1}$, i.e, we have
$$0\rightarrow \mathfrak{n}\xrightarrow{p}
E\xrightarrow{q}\mathfrak{m}\rightarrow 0.$$ We fix two non-zero
vectors $\bv_1'$ and $\bv_2'$ in $\mathfrak{n}$ and $\mathfrak{m}$
respectively. Fix a $Z_k$ equivariant section
$s:\mathfrak{m}\rightarrow E$ of the map $q$. Let $\bv_1$ and $\bv_2$
be the vectors $p(\bv_1')$ and $s(\bv_2')$. For any
$s\in S^{\text{aff}}$ let the action of $T_{s}$ on $\bv_2$ be
\begin{align*}
T_{s}\bv_2&=a_{s}\bv_1-\bv_2\ \forall \ s\in I_2,\\
T_{s}\bv_2&=a_{s}\bv_1, \ \forall \ s\not\in I_2.
\end{align*}
Moreover for any $t$ in $ Z_k$ we have,
\begin{align*}
\lambda_2(t^s)(a_s\bv_1-\bv_2) =T_sT_{t^s}\bv_2
  &=T_tT_s\bv_2=a_s\lambda_1(t)\bv_1-\lambda_2(t)\bv_2 \ \forall \
    s\in I_2, \\
a_s\lambda_2(t^s)\bv_1=\lambda_2(t^s)T_s\bv_2=T_sT_{t^s}\bv_2
  &=T_tT_s\bv_2=\lambda(t)a_s\bv_1, \ \forall \ s\not\in I_2.
\end{align*} 
From the above relations we get that
$a_s((s\lambda_2)(t)-\lambda_1(t))=0$ for all $s\in S_{\text{aff}}$
and $t\in Z_k$. Let $I_E$ be the set
$\{s\in S_{\text{aff}}\ | \ a_s\neq 0\}$. If $E$ is non-split
extension then the set $I_E$ is non-empty and moreover we have
$\lambda_2^s=\lambda_1$ for all $s\in I_E$. If
$\lambda_1\neq \lambda_2$ the values $(a_s)_{s\in S^{\text{aff}}}$ are
determined by $E$ and does not depend on the choice of
$s:\mathfrak{m}\rightarrow E$. If $\lambda_1=\lambda_2$ the section
$s:\mathfrak{m}\rightarrow E$ is not unique we have to take this into
consideration to identify the space of extensions. But for the present
purpose let us fix a section $s:\mathfrak{m}\rightarrow E$.
\begin{lemma}\normalfont\label{quadratic_relations}
For a fixed basis $(\bv_1, \bv_2)$ as above we get that $a_s=0$ for
any $s\in I_1\cap I_2$. If $s\in S_{\lambda_1}$, $s\not\in I_1$
and $s\not\in I_2$ then $a_s=0$. If $s\not\in S_{\lambda_1}$ and $s\in
I_2$ then $a_s=0$.  With the above relations on $a_s$ the
quadratic relations are satisfied for all $s\in S_{\text{aff}}$. 
\end{lemma}
\begin{proof}
  In this lemma we only use the quadratic relations on the elements
  $T_s$ for $s\in S_{\text{aff}}$. To begin with consider any
  $s\in I_1\cap I_2$. Consider the case where $s\in I_1\cap I_2$. In
  this case we have $T_s\bv_2=a_s\bv_1-\bv_2$ and
  $T_s\bv_1=-\bv_1$. Now the quadratic relation on $T_s$ gives us
$$c_s(-a_s\bv_1+\bv_2)=-c_sT_s\bv_2=T_s^2\bv_2=-2a_s\bv_1-\bv_2.$$
Now the reflection $s$ belongs to $I_1$ and $I_2$ and hence
$c_s\bv_1=c_s\bv_2=1$. This concludes that $a_s=0$. Now consider the
case where $s\not\in I_1$ and hence $s\not\in I_2$. In this case we
get that $T_s\bv_1=0$ and $T_s\bv_2=a_s\bv_1$. The quadratic relations
gives us $-c_sT_s\bv_2=0$. Hence we get that $-a_sc_s\bv_1=0$. Now if
$s\in S_{\lambda_1}$ then $c_s\bv_1=\bv_1$ which implies that $a_s=0$.

Consider the case where $s\in I_1$ and $s\not\in I_2$. We have
$T_s\bv_1=-\bv_1$ and $T_s\bv_2=a_s\bv_1$. Now
$T_s(T_s\bv_2)=-a_s\bv_1$. The element
$-c_sT_s\bv_2=-a_sc_s\bv_1$. Since $s\in I_1\subset S_{\lambda_1}$ we
get that $c_s\bv_1=\bv_1$. This shows that $T_s$ satisfies the
required quadratic relation. Now consider the case $s\in I_2$ and
$s\not\in I_1$. In this case we have $T_s\bv_2=a_s\bv_1-\bv_2$ and
$T_s\bv_1=0$. This shows that $T_s(T_s\bv_2)=-T_s\bv_2$. But
$-c_sT_s\bv_2$ is qual to $-a_sc_s\bv_1+c_s\bv_2$.  Now
$s\in I_2\subset S_{\lambda_2}$ we get that $c_s\bv_2=\bv_2$. Hence we
get that $a_sc_s\bv_1=a_s\bv_1$. Now $s\in S_{\lambda_1}$ then the
quadratic relation is satisfied but otherwise $a_s=0$.
\end{proof}
\begin{remark}\label{bridge}\normalfont
  Note that $a_s$ may be non-zero in either of the following cases. In
  the first case, $s\in I_1$ and $s\not\in I_2$, in the second case
  $s\in S_{\lambda_1}$, $s\not\in I_1$ and $s\in I_2$, the third case
  when $s\not\in S_{\lambda_1}$, $s\not\in I_2$. Now the existence of
  extensions and their isomorphism classes can be computed by
  examining the braid relations.
\end{remark}
\begin{lemma}\label{first_braid}\normalfont
For any $s_1$ and $s_2$ in $I_1\backslash (I_1\cap I_2)$ the constants
$a_{s_1}$ equals to $a_{s_2}$. If $s_1$ and $s_2$ belong to
$I_2\backslash (I_1\cap I_2)$ then $a_{s_1}=a_{s_2}$. 
\end{lemma}
\begin{proof}
  In the first case the action of $T_{s_1}$ and $T_{s_2}$ is given by
  $T_{s_1}\bv_1=-\bv_1$, $T_{s_1}\bv_2=a_{s_1}\bv_1$ and
  $T_{s_2}\bv_1=-\bv_1$ and $T_{s_2}\bv_2=a_{s_2}\bv_1$. Now
  $(T_{s_i}T_{s_j})^m\bv_2=-a_{s_j}\bv_1$ and
  $T_{s_j}(T_{s_i}T_{s_j})^m\bv_2=a_{s_j}\bv_1$. Now by braid relation
  for $T_{s_1}$ and $T_{s_2}$ we get that $a_{s_i}=a_{s_j}$. In the
  second case the action of $T_{s_1}$ and $T_{s_2}$ on $E$ are given
  by $T_{s_i}\bv_1=0$ and $T_{s_i}\bv_2=a_{s_i}\bv_1-\bv_2$. Now
  $(T_{s_i}T_{s_j})^m\bv_2=-a_{s_i}\bv_1+\bv_2$ and
  $T_{s_j}(T_{s_i}T_{s_j})^m\bv_2=a_{s_j}\bv_1-\bv_2$. Using braid
  relations we get that $a_{s_i}=a_{s_j}$.
\end{proof}
Note that in the third case when $s\not\in S_{\lambda_1}$ and
$s\in I_2$, $a_s=0$. Now if there is an $s$ in the third case we get
that for any $s\in I_2\backslash (I_1\cap I_1)$ the value $a_s=0$. The
next lemma concludes the verification of the remaining braid relations
on $T_{s}$ for $s\in S_{\text{aff}}$.
\begin{hypothesis}
We make the following hypothesis on the function
$a:S^{\text{aff}}\rightarrow \bar{k}$ sending $s\mapsto a_s$. 
\begin{enumerate}\normalfont\label{conditions}
\item the conditions on $a_{s}$ satisfied by Lemmas
  \ref{quadratic_relations} and \ref{first_braid},
\item If there exist $s_1\in I_1\backslash (I_1\cap I_2)$ and
  $s_2\in I_2\backslash (I_1\cap I_2)$ such that the order of $s_1s_2$
  is $2$ then we have $a_{s_i}+a_{s_j}=0$ for all
  $s_i\in I_1\backslash (I_1\cap I_2)$ and
  $s_j\in I_2\backslash (I_1\cap I_2)$,
\item Let $s\in S^{\text{aff}}\backslash (I_1\cup I_2)$ and there
  exists an element $s'\in I_1\backslash (I_1\cap I_2)$ such that the
  order of $ss'$ is $2$ then $a_s=0$,
\item Let $s\in S^{\text{aff}}\backslash (I_1\cup I_2)$ and there
  exists an element $s'\in I_2\backslash (I_1\cap I_2)$ such that the
  order of $ss'$ is $2$ then $a_s=0$,
\item Let $s\in S^{\text{aff}}\backslash (I_1\cup I_2)$ and there
  exists an element $s'\in  (I_1\cap I_2)$ such that the
  order of $ss'$ is $3$ then $a_s=0$.
\end{enumerate}
\end{hypothesis}
\begin{lemma}\label{final_braid}\normalfont
  Let $E$ be a $2$ dimensional vector space and for any basis
  $(\bv_1, \bv_2)$ of $E$ such that $T_t\bv_1=\lambda_1(t)\bv_1$ and
  $T_t\bv_2=\lambda_2(t)\bv_2$ for all $t\in Z_k$. Suppose the
  function $a:S^{\text{aff}}\rightarrow \bar{k}$ satisfy the following
  Hypothesis \ref{conditions}. The
  relations $T_s\bv_1=-\bv_1$ for $s\in I_1$, $T_s\bv_1=0$ for
  $s\not\in I_1$ and $T_s\bv_2=a_s\bv_1-\bv_2$ for $s\in I_2$,
  $T_s\bv_2=a_s\bv_1$ for $s\not\in I_2$ makes $E$ a $\mathcal{H}$
  module.
\end{lemma}
\begin{proof}
  From Lemma \ref{quadratic_relations} we have to verify braid
  relations in $\mathcal{H}$. Note that Lemma \ref{first_braid} we get
  the braid relations for pairs $(s_1, s_2)$ such that $s_1, s_2$ are
  both in $I_1\backslash (I_1\cap I_2)$ and
  $I_2\backslash (I_1\cap I_2)$. If $s_1$ and $s_2$ both belong to
  $I_1\cap I_2$ then $a_s=0$ and hence braid relations follow as $E$
  is a direct sum when restricted to the algebra generated by
  $T_{s_1}$ and $T_{s_2}$.  Let $s_1$ and $s_2$ belong to
  $S_{\text{aff}}\backslash (I_1\cup I_2)$. In this case we have:
  $T_{s_1}\bv_1=0$, $T_{s_1}\bv_2=a_{s_1}\bv_1$, $T_{s_2}\bv_1=0$ and
  $T_{s_2}\bv_2=a_{s_2}\bv_1$. Now $(T_{s_i}T_{s_j})^m=0$ for
  $m\geq 1$ from which the braid relations follow.

  Fix any $s_1\in I_1\backslash (I_1\cap I_2)$ and consider the case
  when $s_2\in I_1\cap I_2$ then we have $T_{s_1}\bv_1=-\bv_1$,
  $T_{s_1}\bv_2=a_{s_1}\bv_1$, $T_{s_2}\bv_1=-\bv_1$ and
  $T_{s_{2}}\bv_2=-\bv_2$. In this case we have
  $(T_{s_i}T_{s_j})^m\bv_2=-a_{s_1}\bv_1$ and
  $T_{s_j}(T_{s_i}T_{s_j})^m\bv_2=a_{s_1}\bv_1$. Hence the braid
  relations are satisfied. Now consider the case where
  $s_2\in I_2\backslash (I_1\cap I_2)$. In this case we have the
  relations $T_{s_2}\bv_1=0$ and
  $T_{s_2}\bv_2=a_{s_2}\bv_1-\bv_2$. With these relations we get that
  $T_{s_2}T_{s_1}\bv_2=0$,
  $T_{s_1}T_{s_2}\bv_2=-(a_{s_1}+a_{s_2})\bv_1$ and
  $T_{s_2}T_{s_1}T_{s_2}\bv_2=0$. By Hypothesis \ref{conditions}, (2)
  we get the braid relations in this case. Now consider the case when
  $s_1\in I_1\backslash (I_1\cap I_2)$ and
  $s_2\in S_{\text{aff}}\backslash (I_1\cup I_2)$. In this case we
  have $T_{s_2}\bv_1=0$ and $T_{s_2}\bv_2=a_{s_2}\bv_1$. Moreover
  $T_{s_2}T_{s_1}\bv_2=0$, $T_{s_1}T_{s_2}\bv_2=-a_{s_2}\bv_1$ and
  $T_{s_2}T_{s_1}T_{s_2}\bv_2=0$. By Hypothesis \ref{conditions}, (3)
  we get the braid relations in this case.

  Now fix any $s_1\in (I_1\cap I_2)$. We have $T_{s_1}\bv_1=-\bv_1$
  and $T_{s_1}\bv_2=-\bv_2$. Consider the case where
  $s_2\in I_2\backslash (I_1\cap I_2)$. In this case we have
  $T_{s_2}\bv_1=0$ and $T_{s_2}\bv_2=a_{s_2}\bv_1-\bv_2$. In this case
  we have $(T_{s_i}T_{s_j})^m=-a_{s_2}\bv_1+\bv_2$ and
  $T_{s_j}(T_{s_i}T_{s_j})^m=a_{s_1}\bv_1-\bv_2$ and hence the braid
  relations are satisfied. Consider the case when
  $s_2\in S_{\text{aff}}\backslash (I_1\cup I_2)$. In this case we
  have $T_{s_2}\bv_1=0$ and $T_{s_2}\bv_2=a_{s_2}\bv_1$. Which gives
  the relations
  $T_{s_1}T_{s_2}\bv_2=T_{s_2}T_{s_1}\bv_2=-a_{s_2}\bv_1$,
  $T_{s_2}T_{s_1}T_{s_2}\bv_2=0$ and
  $T_{s_1}T_{s_2}T_{s_1}\bv_2=a_{s_2}\bv_1$. By Hypothesis
  \ref{conditions}, (5) we get the braid relations in this case.

Finally we have to consider the case where
$s_1\in I_2\backslash (I_2\cap I_1)$ and
$s_2\in S_{\text{aff}}\backslash (I_1\cup I_2)$. In this case we have
$T_{s_1}\bv_1=0$, $T_{s_1}\bv_2=a_{s_1}\bv_1-\bv_2$, $T_{s_2}\bv_1=0$
and $T_{s_2}\bv_2=a_{s_2}\bv_1$. This shows that
$T_{s_1}T_{s_2}\bv_2=0$, $T_{s_2}T_{s_1}\bv_2=-a_{s_2}\bv_1$ and
$T_{s_1}T_{s_2}T_{s_1}\bv_2=0$. By Hypothesis \ref{conditions}, (4)
  we get the braid relations in this case.
\end{proof}

We will investigate the structure constants with respect to Baer sum.
Let $\mathfrak{m}$ and $\mathfrak{n}$ be the $\mathcal{H}$ modules
$\xi_{\lambda_1, I_1}$ and $\xi_{\lambda_2, I_2}$.  Assume that
$\lambda_1\neq \lambda_2$. Fix a basis vectors
$\bv_1^0\in \mathfrak{n}$ and $\bv_2^0\in \mathfrak{m}$. There is a
canonical basis of $E$ given by $p(\bv_1^1)$ and $q(\bv_2^1)$ and we
denote them by $\bv_1$ and $\bv_2$.  Consider any two extensions
$\mathfrak{n}\hookrightarrow E_1 \twoheadrightarrow \mathfrak{m}$ and
$ \mathfrak{n}\hookrightarrow E_2 \twoheadrightarrow \mathfrak{m}$ and
the Baer sum $E_1\dotplus E_2$ is given by the following commutative
diagram:
\begin{equation}\label{baer_sum}
\begin{tikzpicture}[node distance=2.5cm,auto]
\node(A_1){$0$};
\node(B_1)[right of=A_1]{$\mathfrak{n}\oplus\mathfrak{n}$};
\node(C_1) [right of=B_1] {$E_1\oplus E_2$};
\node(D_1) [right of=C_1] {$\mathfrak{m}\oplus\mathfrak{m}$};
\node(E_1) [right of=D_1] {$0$};
\node(A_2)[below of =A_1]{$0$};
\node(B_2)[below of=B_1]{$\mathfrak{n}\oplus\mathfrak{n}$};
\node(C_2) [below of=C_1] {$E'$};
\node(D_2) [below of=D_1] {$\mathfrak{m}$};
\node(E_2) [right of=D_2] {$0$};
\node(A_3)[below of =A_2]{$0$};
\node(B_3)[below of=B_2]{$\mathfrak{n}$};
\node(C_3)[below of =C_2]{$E_1\dotplus E_2$};
\node(D_3)[below of =D_2]{$\mathfrak{m}$};
\node(E_3)[below of =E_2]{$0$};
\draw[->] (A_1) -- (B_1);
\draw[->] (B_1) to node[above]{$p_1\oplus p_2$}(C_1);
\draw[->](C_1) to node[above]{$q_1\oplus q_2$} (D_1);
\draw[->](D_1)--(E_1);
\draw[->](A_2) -- (B_2);
\draw[->](B_2) to node[above]{$p'$}(C_2);
\draw[->] (C_2) to node[above]{$q'$}(D_2);
\draw[->](D_2) --(E_2);
\draw[->](A_3) -- (B_3);
\draw[->](B_3) to node[above]{$p_3$}(C_3);
\draw[->] (C_3) to node[above]{$q_3$}(D_3);
\draw[->](D_3)--(E_3);
\draw[->](B_2) to node{$f_1=\id$}(B_1);
\draw[->](B_2) to node{$\Sigma$}(B_3);
\draw[->](C_2) to node{$f_2$}(C_1);
\draw[->](C_2) to node{$g_2$}(C_3);
\draw[->](D_2) to node{$\Delta$}(D_1);
\draw[->](D_2) to node{$g_3=\id$}(D_3);
\end{tikzpicture} \end{equation}
Here $\Delta$ and $\Sigma$ are diagonal and the sum maps
respectively. The two rows are pullback and push-out by $\Delta$ and
$\Sigma$ respectively.

We denote by $a_s$, $a_s'$ and $a_s''$ by the structure constants of
$E_1$, $E_2$ and $E_1\dotplus E_2$ respectively. Let $\bv_2'$ be a
vector in $E'$ pulled back via $q'$. Since $\lambda_1\neq \lambda_2$
the vector $\bv_2'$ is unique. Let $f_2(\bv_2')=(\bv_2^1+\bv_2^2)$ for
$\bv_1^1$ and $\bv_2^2$ in the first and second summand of
$E_1\oplus E_2$. Let $\bv_1'$ and $\bv_1''$ be two vectors in each
summand of $\mathfrak{n}\oplus \mathfrak{n}$. Now
$T_{s}\bv_2'=b_s\bv_1'+d_s\bv_1''-\delta_{I_2}(s)\bv_2$ and
$$f_2(T_s\bv_2)=b_sf_2(\bv_1')+c_sf_2(\bv_1'')
-\delta_{I_2}(s)\bv_2^1-\delta_{I_2}(s)\bv_2^2.$$ Comparing the action
of $T_s\bv_2^1$ and $T_s\bv_2^2$ we get that $a_s=b_s$ and
$a_s'=c_s$. Finally considering
$g_2(T_s\bv_2')=(b_s+c_s)\bv_1-\delta_{I_2}\bv_2$ we get that
$a_s''=a_s+a_s'$. This shows that the map sending $E$ to
$(a_s)_{s\in S^{\text{aff}}}$ is injective and surjective onto
functions $(a_s)_{s\in S^{\text{aff}}}$ satisfying the conditions of
Lemmas \ref{quadratic_relations} and \ref{first_braid}.

Now we consider the case where $\lambda_1=\lambda_2$. In this case
there is no canonical basis of $E$ stable under the action of $T_t$
for $t\in Z_k$. If we choose a non-canonical basis then the structure
constants are determined only upto translation by a certain
function. Fix two vectors $\bv_1$ and $\bv_2$ of $\mathfrak{n}$ and
$\mathfrak{m}$ respectively. Now choose a section
$s:\mathfrak{m}\rightarrow E$ of the map $q$. Let $\bv'$ be the vector
$s(\bv_2)$. Now we note that
$T_s\bv'=a_s\bv_1-\delta_{I_2}(s)\bv'$. Now if $s'$ is another section
of the map $q$ then image of $s-s'$ is contained in $\mathfrak{n}$
hence $s'(\bv_2)=s(\bv_2)+k\bv_1$. Let $\bv''=s'(\bv_2)$ and we have
$T_s\bv''=(a_s-k\delta_{I_2}(s)+k\delta_{I_1}(s))\bv_1-\delta_{I_2}(s)\bv''$.
Hence the map sending $E$ to $s\mapsto a_s$ gives a map from
$\Ext^1_{\mathcal{H}}(\mathfrak{m}, \mathfrak{n})$ to functions on
$S_{\text{aff}}$, denoted by $\bar{k}^{S_{\text{aff}}}$, modulo the
function spanned by $\delta_{I_2}-\delta_{I_1}$. We denote by $\theta$
this map
\begin{equation}\label{strange_homology}
\theta:\Ext^1_{\mathcal{H}}(\xi_{\lambda_1, I_1}, \xi_{\lambda_2, I_2})\rightarrow
 \dfrac{\bar{k}^{S_{\text{aff}}}}{<\delta_{I_2}-\delta_{I_1}>}.
\end{equation}
The map $\theta$ is non-canonical and depends on the choice of $\bv_1$
and $\bv_2$, but these vectors are determined upto a scalar. 
\begin{lemma}
The map $\theta$ is a linear map and moreover is injective. 
\end{lemma} 
\begin{proof}
  Let us fix a section $s:\mathfrak{m}\rightarrow E'$ of the map $q$
  in \eqref{baer_sum}. We also fix sections
  $s_i:\mathfrak{m}\rightarrow E_i$ of $q_i$.  This also gives a
  section $s_1$ to the map $q_3$. Let $\bv'_2=s(\bv_2)$ and we denote
  by $\bv_2^1$ and $\bv_2^2$ vectors in each summand of
  $E_1\oplus E_2$ such that $f_2(\bv_2')=\bv_2^1+\bv_2^2$.  Let
  $\bv_1'$ and $\bv_1''$ be two vectors in each summand of
  $\mathfrak{n}\oplus \mathfrak{n}$. Now
  $T_{s}\bv_2'=b_sp'(\bv_1')+d_sp'(\bv_1'')-\delta_{I_2}(s)\bv_2$ and
$$f_2(T_s\bv_2)=b_sf_2(\bv_1')+c_sf_2(\bv_1'')
-\delta_{I_2}(s)\bv_2^1-\delta_{I_2}(s)\bv_2^2.$$ Now the vectors
$\bv_2^1$ and $\bv_2^2$ differ from $s_1(\bv_2)$ and $s_2(\bv_2)$ by
$k_1p_1(\bv_1)$ and $k_2p_2(\bv_1)$. Comparing the action of
$T_s\bv_2^1$ and $T_s\bv_2^2$ we get that
$a_s=b_s+k_1(\delta_{I_2}-\delta_{I_1})$ and
$a_s'=c_s+k_2(\delta_{I_2}-\delta_{I_1})$. This shows that
$g_2(T_s\bv_2')$ is equal to the difference of
$(b_s+c_s)\bv_1-\delta_{I_2}\bv_2$ by
$(k_1+k_2)(\delta_{I_2}-\delta_{I_1})$. Hence the map $\theta$ is
linear map. The injectivity is clear from the definition since the
vanishing of the function $s\mapsto a_{s}$ for all
$s\in S_{\text{aff}}$ implies $E$ splits.
\end{proof}
With this we are ready to state the main result of this article. We
introduce some notations for the main results. Let
$I(\lambda_1, \lambda_2)$ be the subset of $S_{\text{aff}}$ such that
$\lambda_1^{s}=\lambda_2$. Let $I(\lambda_1, I_2)$ be the intersection
of $I(\lambda_1, \lambda_2)$ and
$$ \{s\in S_{\text{aff}}\backslash
(S_{\lambda_1}\cup I_2)\ |\ s \ \text{does not satisfy the Hypothesis
}\ \ref{conditions}(3)\ (4)\ \text{and}\ (5)\}.$$ Note that for any
$s\not\in I(\lambda_1, \lambda_2)$ we know that $a_s=0$. Let
$\delta_1=1$ if $I_1\backslash (I_1\cap I_2)\neq \emptyset$ and
$I_1\backslash (I_1\cap I_2)\subset I(\lambda_1, \lambda_2)$ if
otherwise, we set $\delta_1=0$. We set $\delta_2=1$ if
$I_2\subset S_{\lambda_1}$, $I_2\not\subset I_1$ and
$I_2\backslash (I_1\cap I_2)\subset I(\lambda_1, \lambda_2)$ if
otherwise we set $\delta_2=0$.
\begin{theorem}\label{main_theorem}
  Assume that $\lambda_1\neq \lambda_2$ and $I_1\neq I_2$ then the
  dimension of the space
  $\Ext^1_{\mathcal{H}}(\xi_{\lambda_1, I_1}, \xi_{\lambda_2, I_2})$
  is $|I(\lambda_1, I_2)|+\delta_1+\delta_2$ if $I_1$ and $I_2$ does
  not satisfy the Hypothesis \ref{conditions} (2) and is
  $\Ext^1_{\mathcal{H}}(\xi_{\lambda_1, I_1}, \xi_{\lambda_2, I_2})$
  is $|I(\lambda_1, I_2)|+\delta_1+\delta_2-1$ otherwise. If
  $\lambda_1\neq \lambda_2$ and $I_1=I_2$ then the dimension of the
  space
  $\Ext^1_{\mathcal{H}}(\xi_{\lambda_1, I_1}, \xi_{\lambda_2, I_2})$
  is equal to $|I(\lambda_1, I_2)|$.
  
  Assume that $\lambda_1=\lambda_2$ and $I_1\neq I_1$ the dimension of
  the space
  $\Ext^1_{\mathcal{H}}(\xi_{\lambda_1, I_1}, \xi_{\lambda_2, I_2})$
  is equal to $|I(\lambda_1, I_2)|+\delta_1+\delta_2-1$ if $I_1$ and
  $I_2$ does not satisfy the Hypothesis \ref{conditions} (2) and
  $|I(\lambda_1, I_2)|$ otherwise. Now if $\lambda_1=\lambda_2$ and
  $I_1=I_2$ then the dimension of
  $\Ext^1_{\mathcal{H}}(\xi_{\lambda_1, I_1}, \xi_{\lambda_2, I_2})$
  is $|I(\lambda_1, I_2)|$.
\end{theorem}
\begin{proof}
  If $\lambda_1\neq \lambda_2$ then we have the map
  $E\mapsto (a_{s})_{s\in S_{\text{aff}}}$ is injective linear map
  onto the image determined by Lemmas \ref{quadratic_relations} and
  \ref{first_braid}. If $\lambda_1=\lambda_2$ then we use the map
  $\theta$ with the same conditions as in Lemmas
  \ref{quadratic_relations} and \ref{first_braid} but now we have to
  quotient the image with span of the function
  $\delta_{I_2}-\delta_{I_1}$.
\end{proof}
\begin{corollary}\normalfont
  Let $\lambda_1$ and $\lambda_2$ be two trivial characters and $I_1$
  and $I_2$ are disjoint and they do not satisfy the Hypothesis
  \ref{conditions} (2) then the dimension of
  $\Ext^1_{\mathcal{H}}(\xi_{\id, I_1}, \xi_{\id, I_2})$ is $1$.
\end{corollary}
\begin{example}
  \normalfont We verify this calculation for
  ${\rm SL}_2(\mathbb{Q}_p)$.  These results are proved by other
  methods in \cite{sl2_arxiv}. Let $\chi_r$ be the character of $Z_k$
  sending $t\mapsto t^r$ for $0\leq r<p-1$. Let $S_{\text{aff}}$ be
  the set $\{s_0, s_1\}$, the generators of the affine Weyl group
  $W$. Here $s_0s_1$ has infinite order hence the only relavent
  conditions in Hypothesis \ref{conditions} is the condition $(1)$. We
  use the notation $\xi_{r, I}$ for the character $\xi_{\chi_r,
    I}$. The characters of the affine Hecke algebra are given by
  $\xi_{r, \emptyset}$ for $0<r<p-1$, $\xi_{0, s_0}$, $\xi_{0, s_1}$,
  $\xi_{0, \emptyset}$ and $\xi_{0, S_{\text{aff}}}$. The set of
  characters $\{\xi_{0, \emptyset}, \xi_{0, S_{\text{aff}}}\}$ are not
  supersingular and rest of the characters are supersingular.  We
  first consider the regular case ($r\neq 0$). Consider the case when
  $\mathfrak{m}=\xi_{r_1, \emptyset}$ and
  $\mathfrak{n}=\xi_{r_2, \emptyset}$. The set
  $I(\chi_{r_1}, \chi_{r_2})\neq \emptyset$ if and only if
  $r_1+r_2=p-1$. In which case
  $I(\chi_{r_1}, \chi_{r_2})=S_{\text{aff}}$. We may and do assume
  that $0<r_i= (p-1)/2$ for $i \in \{1,2\}$. The sets
  $S_{\lambda_1}=I_1=I_2=S_{\lambda_2}=\emptyset$. Hence
  $I(S_{\lambda_1}, I_{\lambda_2})=S_{\text{aff}}$. This shows that
  the space of extensions
  $\Ext^1_{\mathcal{H}}(\xi_{r_1, \emptyset}, \xi_{p-1-r_2,
    \emptyset})$ has dimension $2$. If $r_1=r_2=1$ we have $I_1=I_2$
  and $\lambda_1=\lambda_2$ case of \ref{main_theorem} and hence the
  dimension of the space
  $\Ext^1_{\mathcal{H}}(\xi_{r_1, \emptyset}, \xi_{p-1-r_2,
    \emptyset})$ has dimension $2$.

  Consider the case when $r_1=0$ (the Iwahori--case) then the set
  $I(\chi_{r_1}, \chi_{r_2})\neq \emptyset$ if and only if
  $r_2\in \{p-1, 0\}$. We may assume that $r_2=0$. In this case the
  set $I(\chi_{0}, \chi_{0})$ is $S_{\text{aff}}$. The set
  $S_{\chi_0}=S_{\text{aff}}$. Now consider the case when
  $\mathfrak{m}=\xi_{\chi_0, s_i}$ and
  $\mathfrak{n}=\xi_{\chi_0, s_j}$. The set
  $I(\chi_0, \{s_j\})=\emptyset$ and note that $\delta_1=1$ and
  $\delta_2=1$ if $s_i\neq s_j$. If $s_i=s_j$ then $\delta_1=0$ and
  $\delta_2=0$ since $I_2\not\subset I_1$ condition is not
  satisfied. This shows that the dimension of the space
  $\Ext^1_{\mathcal{H}}(\xi_{0, s_i}, \xi_{0, s_j})=1$ if $i\neq j$
  and is zero otherwise.
\end{example}
\begin{remark}\label{counter_example}\normalfont
  In the case of ${\rm SL}_n$ the $L$-packets are defined by Koziol as
  conjugation by ${\rm PGL}_n(K)$ (see \cite[Definition
  6.4]{koziol_sl_packets}). For $n=3$ the Hypothesis \ref{conditions}
  (2) is not relavent. If $\xi_{\lambda_1, I_1}$ and
  $\xi_{\lambda_2, I_2}$ are in the same $L$-packets then the sets
  $|I_1|=|I_2|$. Now consider the simple example for ${\rm SL}_3$ the
  set $S^{\text{aff}}=\{s_1,s_2,s_3\}$ and assume $I_1=\{s_1\}$ and
  $I_2=\{s_2,s_3\}$. The above corollary shows that extensions exist
  among distinct $L$-packets. The notion of blocks and $L$-packets in
  the supersingular case of higher rank groups are different and the
  relationship is not clear in the higher rank cases.
\end{remark}
\section{Blocks for unramified unitary groups in
  \texorpdfstring{$2$}{} and \texorpdfstring{$3$}{} variables.}
As an application we deduce the extensions of simple supersingular
$\mathcal{H}$ modules of unramified groups $U(2,1)$ and $U(1,1)$. In
these cases we will try to precisely understand the relation between
extensions and $L$-packets. The Iwahori--Hecke module structure of
$U(2,1)$ and $U(1,1)$ are studied by Abdellatif, Koziol-Xu and Koziol
in the articles (see\cite{ramla_thesis}, \cite{koziol_unitary_three}
and \cite{koziol_unitary_two}).

Let $L$ be a unramified quadratic extension of $K$ and $(W, h)$ be a
pair consisting of a $3$-dimensional vector space $W$ over $L$ and $h$
be a non-degenerate hermitian form on $W$. We denote by $k_L$ the
residue field of $L$ which is a quadratic extension of $k$.  Let
$\sch{G}$ be the isometry group scheme over $K$ associated to the pair
$(W, h)$. In this case the maximal $K$-split torus $\sch{T}$ is
isomorphic to $(\mul/K)$. The normaliser of $\sch{Z}$ of $\sch{T}$ is
isomorphic to ${\rm Res}_{L/K}\mul\times \sch{U}(1)(L/K)$. The group
$\sch{Z}_k$ is isomorphic to
${\rm Res}_{k_L/k}\mul\times \sch{U}(1)(k_L/k)$ such that the
determinant map is the second projection
$\sch{Z}_k\rightarrow \sch{U}(1)(k_L/k)$. Let us fix a chamber $C$ and
the set $S^{\text{aff}}=\{s_1,s_2\}$ where $s_1$ and $s_2$ are two
affine reflections in the walls of $C$. The order of $s_1s_2$ is
infinite and hence the relevant conditions in Hypothesis
\ref{conditions} is the condition $(1)$.

For quadratic relations we need to describe the groups $Z_{k, s_1}$
and $Z_{k, s_2}$. By abuse of notation we identify the faces fixed by
$s_i$ with $s_i$. With out loss of generality we assume that
$\sch{G}_{s_1, k}$ is isomorphic to $\sch{U}(2,1)(k_L/k)$ and
$\sch{G}_{s_2,k}$ is isomorphic to
$\sch{U}(1,1)(k_L/k)\times U(1)(k_L/k)$. This shows that the group
$<\sch{U}_{s_1, k}, \overline{\sch{U}}_{s_1,k}>\cap \sch{Z}_k$ is
isomorphic to ${\rm Res}_{k_L/k}\mul$ and
$<\sch{U}_{s_2, k}, \overline{\sch{U}}_{s_2,k}>\cap \sch{Z}_k$ is
isomorphic to $\mul/k$. The group $\sch{Z}_{k,s_2}$ embeds in
$\sch{Z}_k$ and is isomorphic to the first factor. Similarly the
group $\sch{Z}_{k,s_2}$ also embeds into the first factor of
$\sch{Z}_k$. The groups $Z_{s_0,k}\simeq k_L^{\times}$ and
$Z_{s_1,k}\simeq k^{\times}$. Since $\Lambda$ is commutative the group
$W(1)$ acts on $Z_k$ by the quotient
$W(1)\rightarrow W_0\simeq \{\id, s_1\}$.

Let $\zeta:k_E^\times\rightarrow \bar{k}$ and
$\eta:U(1)\rightarrow \bar{k}$ be any two characters then we denote by
$\chi$ the character $\zeta\otimes\eta$. Let $x\mapsto \bar{x}$ be the
nontrivial Galois automorphism on $k_L$. The character $\chi^{s_1}$ is
given by $\overline{\zeta}\otimes\eta$ where
$\overline{\zeta}(x)=\zeta(\bar{x}^{-1})$. Note that the character
$\chi=\zeta\otimes\eta$ is trivial on $Z_{k, s_1}$ if and only if
$\zeta$ is trivial and $\chi$ is trivial on $Z_{k, s_2}$ if and only
if $\zeta^{q+1}=\id$. These cases are called as trivial-Iwahori and
hybrid respectively by Koziol--Xu in these cases $\chi^{s_1}=\chi$. If
$\chi$ is non-trivial on $Z_{k,s_1}$ and $Z_{k, s_2}$ then the
character $\chi$ is called as regular and $\chi^{s_1}\neq \chi$. Now we
list the various characters of
$\mathcal{H}=\mathcal{H}^{\text{aff}}$. We use the description of
supersingular characters given by Vigneras but we point out that these
are also described by Koziol--Xu.
\begin{enumerate}
\item If $\chi=\zeta\otimes\eta$ is a trivial-Iwahori type then we
  have $S_{\chi}=S^{\text{aff}}$. So we have two supersingular
  characters $\xi_{\chi, s_1}$ and $\xi_{\chi, s_2}$. The characters
  $\xi_{\chi, S^{\text{aff}}}$ and $\xi_{\chi, \emptyset}$ are not
  supersingular.
\item If $\chi=\zeta\otimes\eta$ is hybrid then $S_{\chi}=\{s_2\}$ and
  in this case we have two supersingular characters $\xi_{\chi, s_2}$
  and $\xi_{\chi, \emptyset}$,
\item If $\chi=\zeta\otimes\eta$ is a regular character then we have
  only one supersingular character $\xi_{\chi, \emptyset}$,
\end{enumerate}
Note that if $\chi$ is a trivial or hybrid type character then $I(\chi,
\chi')$ is not an empty set if and only if $\chi=\chi'$.   
\begin{proposition}
  Let $\chi$ be a trivial character then the dimension of
  $\Ext^1_{\mathcal{H}}(\xi_{\chi, s_i}, \xi_{\chi, s_j})$ is $1$
  $i\neq j$ and is $0$ otherwise. If $\chi$ is a hybrid character then
  the dimension of
  $\Ext^1_{\mathcal{H}}(\xi_{\chi, I_1}, \xi_{\chi, I_2})$ is $1$ for
  all $I_1, I_2\subset \{s_2\}$. If $\chi$ is regular the dimension of
  the space
  $\Ext^1_{\mathcal{H}}(\xi_{\chi, \emptyset}, \xi_{\chi',
    \emptyset})$ is $2$ when $\chi^{s_0}=\chi'$ and zero otherwise.
\end{proposition}
\begin{proof}
  Let $\xi_{\lambda_1, I_1}$ and $\xi_{\lambda_2, I_2}$ be the
  characters $\xi_{\chi, s_i}$ and $\xi_{\chi, s_j}$ respectively. We
  observed that $I(\chi, \chi)=S^{\text{aff}}$ and we have
  $I(\chi, \{s_i\})=\emptyset$. If $i\neq j$ then we have
  $I_2\not\subset I_1$ and hence $\delta_2=1$ and
  $I_1\backslash (I_1\cap I_2)$ is nonempty which gives
  $\delta_1=1$. If $i=j$ we have $\delta_1=0$. The subsets $I_1=I_2$
  hence $\delta_2=0$. Applying Theorem \ref{main_theorem} we get that
  the dimension of the extension space
  $\Ext^1_{\mathcal{H}}(\xi_{\chi, s_i}, \xi_{\chi, s_j})$ is $1$ if
  $i\neq j$ and the dimension of
  $\Ext^1_{\mathcal{H}}(\xi_{\chi, s_i}, \xi_{\chi, s_i})$ is $0$.
 
  Assume that $\chi$ is hybrid character. Let $\xi_{\lambda_1, I_1}$
  and $\xi_{\lambda_2, I_2}$ be the characters $\xi_{\chi, s_2}$ and
  $\xi_{\chi, \emptyset}$ respectively. We note that
  $I(\lambda_1, \lambda_2)=S^{\text{aff}}$, the set
  $I(\lambda_1, I_2)=\{s_1\}$. Now if $I_1=\{s_2\}$ and $I_2=\{s_2\}$
  we have $\delta_1=0$ and $\delta_2=0$. The dimension of the space
  $\Ext^1_{\mathcal{H}}(\xi_{\chi, s_2}, \xi_{\chi, s_2})$ is
  $1$. Assume that $I_1=\{\emptyset\}$ and $I_2=\{s_2\}$ then
  $\delta_1=0$ and $\delta_2=1$ and hence the dimension of
  $\Ext^1_{\mathcal{H}}(\xi_{\chi, s_2}, \xi_{\chi, \emptyset})$ is
  $1$.  If $I_1=\{s_2\}$ and $I_2=\{\emptyset\}$ then we have
  $\delta_1=1$ and $\delta_2=0$ hence the dimension of the space
  $\Ext^1_{\mathcal{H}}(\xi_{\chi, s_2}, \xi_{\chi, \emptyset})$ is
  $1$. Finally assume that $I_1=I_2=\emptyset$ in this case
  $\delta_1=0$ and $\delta_2=\emptyset$ the dimension of
  $\Ext^1_{\mathcal{H}}(\xi_{\chi, \emptyset}, \xi_{\chi, \emptyset})$
  is $1$.

  Finally we consider the case when $\chi$ is regular. Assume that
  $\xi_{\lambda_1, I_1}$ and $\xi_{\lambda_2, I_2}$ be the characters
  $\xi_{\lambda_1, \emptyset}$ and $\xi_{\lambda_2, \emptyset}$ and
  assume that $I(\lambda_1, \lambda_2)\neq \emptyset$. In this case we
  have $S_{\lambda_1}=I_1=I_2=S_{\lambda_2}=\emptyset$ and
  $|I(\lambda_1, I_2)|=2$. Moreover we have $\delta_1=0$ and
  $\delta_2=0$ and the dimension of the space
  $\Ext^1_{\mathcal{H}}(\xi_{\lambda_1, \emptyset}, \xi_{\lambda_2,
    \emptyset})$ is $2$.
\end{proof}
Now we will consider the case of unitary group $U(1,1)(L/K)$ where $L$
is unramified over $K$. Let $(W, h)$ be a pair consisting of a $2$
dimensional vector space over $L$ and $h$ be a non-degenerate
hermitian form on $W$. Let $\sch{U}(1,1)$ be the unitary group scheme
over $F$ attached to $(W, h)$. In this case the maximal $K$-split
torus $\sch{T}$ is isomorphic to $\mul/K$, its normalier is isomorphic
to ${\rm Res}_{L/K}\mul$. Hence the group $\sch{Z}_k$ is isomorphic to
${\rm Res}_{k_L/k}\mul$. Lets fix a chamber $C$ and the set
$S^{\text{aff}}$ is given by $\{s_1, s_2\}$. The group schemes
$\sch{G}_{k, s}\simeq \sch{U}(1,1)(k_L/k)$ for $s\in \{s_1,s_2\}$ as
$s_1$ and $s_2$ are conjugate in $GU(1,1)$. This shows that
$\sch{Z}_{k, s_0}$ and $\sch{Z}_{k, s_1}$ are both isomorphic to
$\mul/k$ as $<\sch{U}_{k,s}, \overline{\sch{U}}_{k,s}>={\rm SL}_2/k$.

Now consider any character $\chi$ on $k_E^{\times}$ and the character
$\chi$ is trivial on $Z_{k,s_i}$ if and only if
$\chi^{q+1}=\id$. Following the above case we use the terminology that
$\chi$ is trivial type if $\chi$ is trivial and hybrid type if
$\chi^{q+1}=1$ and $\chi\neq \id$. We also note that $\chi^{s_1}=\chi$
if and only if $\chi^{q+1}=\id$. Now the characters of the algebra
$\mathcal{H}=\mathcal{H}^{\text{aff}}$ are given by
\begin{enumerate}
\item If $\chi$ is either trivial or hybrid character then
  $S_{\chi}=S^{\text{aff}}$ and $\xi_{\chi, s_1}$ and
  $\xi_{\chi, s_2}$ are supersingular characters. The characters
  $\xi_{\chi, \emptyset}$ and $\xi_{\chi, S^{\text{aff}}}$ are not
  supersingular characters.
\item If $\chi$ is a regular character then $S_{\chi}=\emptyset$ and
  the only characters of $\mathcal{H}$ are $\xi_{\chi, \emptyset}$.
\end{enumerate}
To begin with the calculation of extensions, for a character
$\chi:k_L^{\times}\rightarrow R^\times$ such that $\chi^{q+1}=\id$,
the set $I(\chi, \chi')\neq \emptyset$ if and only if
$(\chi')^{q+1}=\id$.
\begin{proposition}
  Let $\chi$ be a character such that $\chi^{q+1}=\id$ then the
  dimension of the space
  $\Ext^1_{\mathcal{H}}(\xi_{\chi, s_i}, \xi_{\chi, s_j})$ is $1$ if
  $i\neq j$ and is zero otherwise. If $\chi$ is a regular character
  then the dimension of the space
  $\Ext^1_{\mathcal{H}}(\xi_{\chi, \emptyset}, \xi_{\chi',
    \emptyset})$ is $2$ if $\chi'=\chi^{s_1}$ and is trivial
  otherwise.
\end{proposition}
\begin{proof}
  This situation is similar to ${\rm SL}_2$. Assume that
  $\chi^{q+1}=\id$ and $\xi_{\lambda_1, I_1}$ and
  $\xi_{\lambda_2, I_2}$ be the characters $\xi_{\chi, s_i}$ and
  $\xi_{\chi, s_j}$. In this case we have
  $I(\lambda_1, I_2)=S^{\text{aff}}$ has cardinality two. If $i\neq j$
  we have $\delta_1=1$ and $\delta_2=1$ and since $I_1\neq I_2$ we get
  that the dimension of
  $\Ext^1_{\mathcal{H}}(\xi_{\chi, s_i}, \xi_{\chi, s_j})$ is $1$ if
  $i\neq j$ and is zero otherwise. Assume that $\chi$ and
  $\xi_{\lambda_1, I_1}$ and $xi_{\lambda_2, I_2}$ be the characters
  $\xi_{\chi, \emptyset}$ and $\xi_{\chi^{s_1}, \emptyset}$. The set
  $I(\lambda_1, \lambda_2)\neq \emptyset$ if and only if
  $\lambda_1=\lambda_2^{s_1}$. If
  $ I(\lambda_1, \lambda_2)\neq \emptyset$ then we have
  $I(\lambda_1, \lambda_2)=S^{\text{aff}}$.  The sets
  $I(\lambda_1, \lambda_2)=\emptyset$,
  $S_{\lambda_1}=I_1=I_2=S_{\lambda_2}=\emptyset$. This shows that
  $\delta_1=0$ and $\delta_2=0$. The set $I(\lambda_1, I_2)$ has
  cardinality $2$. This shows that
  $\Ext^1_{\mathcal{H}}(\xi_{\chi, \emptyset}, \xi_{\chi^{s_0},
    \emptyset})$ has dimension $2$.
\end{proof}
\bibliography{../biblio} \bibliographystyle{amsalpha}
\noindent Santosh Nadimpalli,\\
School of Mathematics, Tata Institute of Fundamental Research,
Mumbai, 400005.\\
\texttt{nvrnsantosh@gmail.com}, \texttt{nsantosh@math.tifr.res.in}
\end{document}